\documentclass[12pt]{amsart}
\usepackage{amscd}
\def\frk{\frak}               

\def\Phi{{\frk n}}
\def\Phi{{\frk N}}
\def\opn#1#2{\def#1{\operatorname{#2}}} 
\opn\chara{char} \opn\length{\ell} \opn\pd{pd} \opn\rk{rk}
\opn\projdim{proj\,dim} \opn\injdim{inj\,dim} \opn\rank{rank}
\opn\depth{depth} \opn\sdepth{sdepth} \opn\fdepth{fdepth}
\opn\grade{grade} \opn\height{height} \opn\embdim{emb\,dim}
\opn\codim{codim}  \opn\min{min} \opn\max{max}

\opn\Tr{Tr} \opn\bigrank{big\,rank}
\opn\superheight{superheight}\opn\lcm{lcm}
\opn\trdeg{tr\,deg}
\opn\reg{reg} \opn\lreg{lreg} \opn\ini{in} \opn\lpd{lpd}
\opn\size{size}
\opn\div{div} \opn\Div{Div} \opn\cl{cl} \opn\Cl{Cl}
\opn\Spec{Spec} \opn\Supp{Supp} \opn\supp{supp} \opn\Sing{Sing}
\opn\Ass{Ass} \opn\Min{Min}
\opn\Ann{Ann} \opn\Rad{Rad} \opn\Soc{Soc}
\opn\Im{Im} \opn\Ker{Ker} \opn\Coker{Coker} \opn\Am{Am}
\opn\Hom{Hom} \opn\Tor{Tor} \opn\Ext{Ext} \opn\End{End}
\opn\Aut{Aut} \opn\id{id}  \opn\deg{deg}

\opn\nat{nat}
\opn\pff{pf}
\opn\Pf{Pf} \opn\GL{GL} \opn\SL{SL} \opn\mod{mod} \opn\ord{ord}
\opn\Gin{Gin} \opn\Hilb{Hilb}
\opn\aff{aff} \opn\con{conv} \opn\relint{relint} \opn\st{st}
\opn\lk{lk} \opn\cn{cn} \opn\core{core} \opn\vol{vol}
\opn\link{link} \opn\star{star}
\opn\gr{gr}

\def\pot#1#2{#1[\kern-0.28ex[#2]\kern-0.28ex]}

\opn\dirlim{\underrightarrow{\lim}}
\opn\inivlim{\underleftarrow{\lim}}
\let\Dirsum=\bigoplus

\def\Implies{\ifmmode\Longrightarrow \else
        \unskip${}\Longrightarrow{}$\ignorespaces\fi}
\def\implies{\ifmmode\Rightarrow \else
        \unskip${}\Rightarrow{}$\ignorespaces\fi}
\def\iff{\ifmmode\Longleftrightarrow \else
        \unskip${}\Longleftrightarrow{}$\ignorespaces\fi}

\let\:=\colon
\newtheorem{Theorem}{Theorem}[section]
\newtheorem{Lemma}[Theorem]{Lemma}

\newtheorem{Proposition}[Theorem]{Proposition}
\newtheorem{Remark}[Theorem]{Remark}

\newtheorem{Example}[Theorem]{Example}

\let\epsilon\varepsilon
\let\phi=\varphi
\let\kappa=\varkappa
\def\qed{\ifhmode\textqed\fi
      \ifmmode\ifinner\quad\qedsymbol\else\dispqed\fi\fi}
\def\textqed{\unskip\nobreak\penalty50
       \hskip2em\hbox{}\nobreak\hfil\qedsymbol
       \parfillskip=0pt \finalhyphendemerits=0}
\def\dispqed{\rlap{\qquad\qedsymbol}}

\opn\dis{dis}
\def\pnt{{\raise0.5mm\hbox{\large\bf.}}}

\opn\Lex{Lex}

\begin{document}

\title{\bf Special Stanley Decompositions}

\author{ Adrian Popescu }

\thanks{}

\address{Adrian Popescu, University of Bucharest, Department of Mathematics, Str. Academiei 14, Bucharest, Romania}
\email{adi\_popescum@yahoo.de}

\maketitle
\begin{abstract}
Let $I$ be  an intersection of three monomial prime ideals of a
polynomial algebra $S$ over a field. We give a special Stanley
decomposition of $I$ which provides a lower bound of the Stanley
depth of $I$, greater than or equal to $\depth\ (I)$, that is
Stanley's Conjecture holds for $I$.

  \vskip 0.4 true cm
 \noindent
\end{abstract}

\section*{Introduction}

 Let $K$ be a field and $S=K[x_1,\ldots,x_n]$ be the polynomial ring over $K$
 in $n$ variables. Let $I\subset S$ be a squarefree monomial ideal of $S$, $u\in I$ a monomial and $uK[Z]$, $Z\subset \{x_1,\ldots ,x_n\}$ the linear $K$-subspace of $I$ of all elements $uf$, $f\in K[Z]$. Suppose that $uK[Z]$ is free over $K[Z]$. A Stanley decomposition of $I$ is a presentation of $I$ as a finite direct sum of such spaces ${\mathcal D}:\
I=\Dirsum_{i=1}^ru_iK[Z_i]$. Set $\sdepth
(\mathcal{D})=\min\{|Z_i|:i=1,\ldots,r\}$ and
\[
\sdepth\ (I) :=\max\{\sdepth \ ({\mathcal D}):\; {\mathcal D}\; \text{is a
Stanley decomposition of}\;  I \}.
\]

By Stanley's Conjecture \cite{S} the Stanley depth $\sdepth\ (I)$
of $I$ is $\geq \depth\ (I)$. This is proved if either $n\leq 5$
by \cite{P}, or $I$ is the intersection of two monomial
irreducible ideals by \cite[Theorem 5.6]{PQ}. It is the purpose of
our paper to show that Stanley's Conjecture holds for
intersections of three monomial prime ideals  (see Theorem
\ref{sd>d}) and for arbitrary intersections of prime ideals
generated by disjoint sets of variables (see Theorem \ref{th1}).
For the proof we give a special Stanley decomposition ${\mathcal
D}$ of $I$ and compute $\sdepth \ ({\mathcal D})$ (see Lemma
\ref{sd} and Proposition \ref{sdepth}) which is $\geq \depth\
(I)$.

\vskip 1 cm

\section{Intersections of primes generated by disjoint sets of variables}
Let $S=K[x_1,\ldots,x_n]$, $I \subset K[x_1,\ldots,x_r]=S'$ and $J
\subset K[x_{r+1},\ldots,x_n]=S''$ be monomial ideals, where
$1<r<n$. The following two lemmas are elementary, their proofs
being suggested by \cite[Theorem 2.2.21]{Vi} and \cite[Lemma
4.1]{PQ}.
\begin{Lemma}\label{dep1}Then
$$\depth_{S}(IS \cap JS)=\depth_{S'} (I) + \depth_{S''}(J).$$
\end{Lemma}
\begin{proof}
In the exact sequence of $S$-modules:
$$0 \rightarrow \frac{S}{IS \cap JS} \rightarrow \frac{S}{IS} \oplus \frac{S}{JS} \rightarrow \frac{S}{IS+JS} \rightarrow 0$$
 we have
 $\depth_S\left(\frac{S}{IS+JS}\right)=\depth_{S'}(S'/I)+\depth_{S''}(S''/J)$ by \cite[Theorem
 2.2.21]{Vi}. Using Depth Lemma we get $$\depth_S\left(\frac{S}{IS \cap
JS}\right)=\depth_S\left(\frac{S}{IS+JS}\right)+1=\depth_{S'}(S'/I)+\depth_{S''}(S''/J)+1,$$
which is enough.
\end{proof}
The following result is analogous to  \cite[Theorem 3.1]{R}.
\begin{Lemma}\label{sdep1}With the hypotheses from the previous lemma, we have
$$\sdepth_S(IS \cap JS) \ge \sdepth_{S'}(I)+\sdepth_{S''}(J).$$
\end{Lemma}
The proof of \cite[Lemma 4.1]{PQ} works also in our case.

\begin{Remark} {\em The inequality of the above lemma can be strict as happens in the case when $n=4$, $r=2$, $I=(x_1,x_2)$, $J=(x_3,x_4)$. Indeed, then
$\sdepth_{S'}(I)=1$, $\sdepth_{S''}(J)=1$, and $\sdepth_S(IS\cap
JS)=3>2=\sdepth_{S'}(I)+\sdepth_{S''}(J)$, as shows the  Stanley
decomposition $$IS\cap JS=x_1x_3K[x_1,x_3,x_4]\oplus
x_1x_4K[x_1,x_2,x_4]\oplus $$
$$x_2x_3 K[x_1,x_2,x_3]\oplus x_2x_4K[x_2,x_3,x_4]\oplus
x_1x_2x_3x_4S.$$}
\end{Remark}
\begin{Theorem}\label{th1} Let $0=r_0<r_1<r_2<\ldots<r_s=n$, $S=K[x_1,\ldots,x_n]$ and set
$P_1=(x_1,\ldots,x_{r_1})$,
$P_2=(x_{r_1+1},\ldots,x_{r_2}),\ldots,P_s=(x_{r_{s-1}+1},\ldots,x_{r_s})$
and $I=\displaystyle\bigcap_{i=1}^s P_i$. Then,
$$\sdepth\ (I) \ge  \depth\ (I) =s,$$
and in particular Stanley's Conjecture holds in this case.
\end{Theorem}
\begin{proof}Let us denote $S^i=K[x_{r_{i-1}+1},\ldots,x_{r_i}]$. Since
$\depth_{S^i}\left(\displaystyle\frac{S^i}{P_i \cap S^i}\right)=0$
we get $\depth_{S^i}(P_i \cap S^i)=1$. By Lemma \ref{dep1} and
recurrence we obtain $\depth\ (I) = s$. Applying Lemma
\ref{sdep1}, by recurrence and \cite{Sh} we get that
$$\sdepth_S\ (I) \ge \displaystyle\sum_{i=1}^s \sdepth_{S^i}(P_i \cap S^i) = \displaystyle\sum_{i=1}^s
\left\lceil\displaystyle\frac{r_i-r_{i-1}}{2}\right\rceil \ge s,$$ where $\lceil a \rceil$ is the lowest integer number greater or equal to $a \in \mathbb{R}.$
\end{proof}
\begin{Remark}
If $s=2$, then Ishaq \cite[Corollary 2.9, 2.10]{Is} proved that
$\sdepth(I) = \left\lceil\displaystyle\frac{n+1}{2}\right\rceil$
if either $n$ is odd, or $n$ is even but $r_1$ is odd, and
$\displaystyle\frac{n}{2} \le \sdepth(I) \le
\displaystyle\frac{n}{2}+1$ if $n$ and $r_1$ are even.
\end{Remark}

\vskip 1 cm
In the next section  we need the following two
lemmas:

\begin{Lemma}\cite[Lemma 4.3]{PQ}\label{pq} Let $S=K[x_1,\ldots,x_n]$,
$Q=(x_1,\ldots,x_t)$ and $Q'=(x_{r+1},\ldots,x_n)$ where $1\le r
\le t < n$. Then
$$\sdepth(Q \cap Q') \ge \left\lceil\displaystyle\frac{r}{2}\right\rceil+ \left\lceil\displaystyle\frac{n-t}{2}\right\rceil.$$
\end{Lemma}

\begin{Lemma}\cite[Lemma 3.6]{HVZ}\label{hvz} Let $I \subset S=K[x_1,\ldots,x_n]$ and
$S'=S[x_{n+1},\ldots,x_t]$, where $t>n$. Then,
$$\sdepth_{S'}(IS')=\sdepth_S(I) + (t-n).$$
\end{Lemma}

\vskip 1 cm
\section{Intersections of three prime ideals}
Let $S$=K[$x_{1}$,\ldots,$x_{n}$] and $P_{1}$,$P_{2}$,$P_{3}$ be
three non-zero monomial prime ideals not included one in the other
such that $\displaystyle\sum\limits_{i=1}^3
P_{i}=(x_{1},\ldots,x_{n})$. Let I=$P_{1} \cap P_{2} \cap P_{3}.$
\begin{Proposition}\label{depth} Then
\[
\depth\ (I) = \left\{
\begin{array}{l l}
  3,      \textnormal{ if $P_{i} \not\subset P_{j}+P_{k}$ \textnormal{for any different}  $i,j,k \in \{1,2,3\}$}\\
  n+2-\max\{\textnormal{ht}(P_{i}+P_{j}),\textnormal{ht}(P_{i}+P_{k})\}, \textnormal{if $P_{i} \subset P_{j}+P_{k}$ }.\\
\end{array} \right.
\]

\end{Proposition}
\begin{proof}
Consider the following  exact sequence of S-modules:

$$(1) \quad 0 \rightarrow  \frac{S}{P_{1} \cap P_{2}} \rightarrow
\frac{S}{P_{1}} \oplus \frac{S}{P_{2}} \rightarrow
\frac{S}{P_{1}+P_{2}} \rightarrow 0.$$

We have :
 \[\depth \
\left(\frac{S}{P_{1}+P_{2}}\right)=n-\textnormal{ht}(P_{1}+P_{2})\]
and
 \[\depth \
\left(\frac{S}{P_{1}} \oplus \frac{S}{P_{2}}\right)= \min \left \{
\depth (S/P_1),\depth(S/P_2) \right \}=n-\max
\left\{\textnormal{ht}(P_1),\textnormal{ht}(P_{2})\right\}.\] By
hypotheses
$\textnormal{ht}(P_{1}+P_{2})>\max\{\textnormal{ht}(P_{1}),\textnormal{ht}(P_{2})\}$
thus
$n-\textnormal{ht}(P_{1}+P_{2})<n-\max\{\textnormal{ht}(P_{1}),\textnormal{ht}(P_{2})\}.$
That means $\depth(\frac{S}{P_{1}+P_{2}})<\depth(\frac{S}{P_{1}}
\oplus \frac{S}{P_{2}})$ and applying Depth Lemma to (1) we obtain
$\depth(\frac{S}{P_{1} \cap
P_{2}})=n-\textnormal{ht}(P_{1}+P_{2})+1.$ There are two cases:

{\bf Case 1.} $P_{i} \not \subset P_{j}+P_{k}$ for any
    different $i,j,k \in \{1,2,3\}$.

     Let us consider the following two exact sequences of
    S-modules: $$(2) \quad 0 \rightarrow \frac{S}{I}
    \rightarrow \frac{S}{P_{1} \cap P_{2}} \oplus
\frac{S}{P_{3}} \rightarrow \frac{S}{P_{3}+(P_{1} \cap P_{2})}
\rightarrow 0,$$
$$(3) \quad 0 \rightarrow
    \frac{S}{(P_{1}+P_{3})\cap(P_{2}+P_{3})} \rightarrow
\frac{S}{(P_{1}+P_{3})} \oplus \frac{S}{(P_{2}+P_{3})} \rightarrow
\frac{S}{P_{1}+P_{2}+ P_{3}} \rightarrow 0.$$ By the hypothesis of
this case, we have $\depth(\frac{S}{P_{1}+P_{3}})>0$ and
$\depth(\frac{S}{P_{2}+P_{3}})>0$. Applying Depth Lemma to (3) we
get
$\depth\left(\frac{S}{(P_{1}+P_{3})\cap(P_{2}+P_{3})}\right)=1$
because $P_1+P_{2}+P_{3}$ is the maximal ideal. But
$(P_{1}+P_{3})\cap(P_{2}+P_{3})=P_{3}+(P_{1} \cap P_{2})$, so we
get $\depth\left(\frac{S}{P_{3}+(P_{1} \cap P_{2})}\right)=1.$
Using again the hypothesis of this case in (2) we can say that
$\depth(S/P_3)>1$ and $\depth(\frac{S}{P_{1} \cap P_{2}})>1$. By
Depth Lemma applied to (2) we have $\depth(S/I)=2$. Thus $\depth\
(I)=3.$

{\bf Case 2.} There exist different $i,j,k \in \{1,2,3\}$ such
    that $P_{i} \subset P_{j}+P_{k}$.

    After a possible
    renumbering of $(P_i)_{1 \le i \le 3}$ we may suppose that $P_{1} \subset
    P_{2}+P_{3}$. Note that $P_{1}=P_{1} \cap
    P_{2} + P_{1} \cap P_{3}$. Let us consider the next
    exact sequence of $S$-modules:
$$(4) \quad 0 \rightarrow
    \frac{S}{I} \rightarrow \frac{S}{(P_{1} \cap P_{2})}
\oplus \frac{S}{P_{1} \cap P_{3}} \rightarrow \frac{S}{P_{1}}
\rightarrow 0.$$ Remark that $\depth(\frac{S}{P_{1} \cap P_{2}})$
and $\depth(\frac{S}{P_{1} \cap P_{3}})$ are smaller or equal than
$\dim(S/P_1)$ (see \cite{BH}). We prove that
$$\depth(S/I)=\min\{\depth(\frac{S}{P_{1} \cap
P_{2}}),\depth(\frac{S}{P_{1} \cap P_{3}})\}.$$

If
$n-\textnormal{ht}(P_{1})=\dim(S/P_1)>\min\{\depth(\frac{S}{P_{1}
\cap P_{2}}),\depth(\frac{S}{P_{1} \cap
P_{3}})\}=n+1-\max\{\textnormal{ht}(P_{1}+P_{2}),\textnormal{ht}(P_{1}+P_{3})\}$
then we are done by Depth Lemma applied to (4).

Otherwise,
$n-\textnormal{ht}(P_{1})=n+1-\max\{\textnormal{ht}(P_{1}+P_{2}),\textnormal{ht}(P_{1}+P_{3})\}$
and applying again Depth Lemma we get that $\depth(S/I)\ge
\min\{\depth(\frac{S}{P_{1} \cap P_{2}}),\depth(\frac{S}{P_{1}
\cap P_{3}})\}$, the inequality being equality because
$\depth(S/I)\le
\dim(S/P_1)=n-\textnormal{ht}(P_1)=\min\{\depth(\frac{S}{P_{1}
\cap P_{2}}),\depth(\frac{S}{P_{1} \cap P_{3}})\}$. Thus we get
$$\depth(S/I)=n+1-\max\{\textnormal{ht}(P_{1}+P_{2}),\textnormal{ht}(P_{1}+P_{3})\}$$
and so $$\depth\
(I)=n+2-\max\{\textnormal{ht}(P_{1}+P_{2}),\textnormal{ht}(P_{1}+P_{3})\}.$$
\end{proof}

The next lemma presents a decomposition of the above $I$ as a
direct sum of its linear subspaces. These subspaces are
``simpler'' monomial ideals, for which we already know ``good''
Stanley decompositions. Substituting them in the above direct sum
we get some {\em special Stanley decompositions} where it is
easier to lower bound their Stanley depth.

We may suppose after a possible renumbering of variables that
$P_1=(x_1,\ldots,x_r)$. Let us denote the following:

$$b_2 - \textnormal{the number of variables from } \{x_i|1 \le i \le r\} \textnormal{ for
which } x_i \in P_2,$$
$$b_3 - \textnormal{the number of variables from } \{x_i|1 \le i \le r\} \textnormal{ for
which } x_i \in P_3,$$
$$b_1 - \textnormal{the number of variables from } \{x_i|1 \le i \le r\} \textnormal{ for
which } x_i \in P_2 \cup P_3,$$
$$a_{23} - \textnormal{the number of variables from } \{x_i|1 \le i \le r\} \textnormal{ for
which } x_i \in P_2 \setminus P_3,$$
$$a_{32} - \textnormal{the number of variables from } \{x_i|1 \le i \le r\} \textnormal{ for
which } x_i \in P_3 \setminus P_2,$$
$$c - \textnormal{the number of variables from } \{x_i|r+1 \le i \le n\} \textnormal{ for
which } x_i \in P_2 \cap P_3,$$
$$A = \left\lceil \frac{a_{32}}{2}\right\rceil+\left\lceil
\frac{\textnormal{ht}(P_2)-b_2}{2}
\right\rceil+n-a_{32}-\textnormal{ht}(P_2),$$$$B = \left\lceil
\frac{a_{23}}{2}\right\rceil+\left\lceil
\frac{\textnormal{ht}(P_3)-b_3}{2}
\right\rceil+n-a_{23}-\textnormal{ht}(P_3),$$

$$C = \left\lceil \frac{r-b_1}{2}\right\rceil+\left\lceil
\frac{\textnormal{ht}(P_2)-b_2-c}{2} \right\rceil+\left\lceil
\frac{\textnormal{ht}(P_3)-b_3-c}{2} \right\rceil,$$
$$S' = K[\{x_i|1 \le i \le r, x_i \not\in
P_3\}],\ S'' = K[\{x_i|1 \le i \le r, x_i \not\in
P_2\}],$$ $$\widetilde{S} = K[\{x_i|1 \le i \le r, x_i
\not\in P_2+P_3\},x_{r+1},\ldots,x_n].$$

\begin{Lemma}\label{sd} Let $S$=K[$x_{1}$,\ldots,$x_{n}$] and $P_{1}$,$P_{2}$,$P_{3}$ be
three non-zero monomial prime ideals not included one in the other
such that $\displaystyle\sum\limits_{i=1}^3
P_{i}=(x_{1},\ldots,x_{n})$. Let I=$P_{1} \cap P_{2} \cap P_{3}$.
The next sum is a direct sum of linear subspaces of $I$:
$$I = I_1 \oplus I_2 \oplus I_3 \oplus I_4,$$ where:
$$I_1=(I\cap K[x_1,\ldots ,x_r])S,\ \
I_2=(P_2 \cap S')S'[x_{r+1},\ldots,x_n] \cap (P_3 \cap
S'[x_{r+1},\ldots,x_n]),$$ $$I_3=(P_3 \cap
S'')S''[x_{r+1},\ldots,x_n] \cap (P_2 \cap
S''[x_{r+1},\ldots,x_n]),\ \ I_4 = I \cap \widetilde{S}.$$
\end{Lemma}

\begin{proof}Note that $I \supseteq I_1 + I_2 + I_3 + I_4$ is obvious because
every $I_i \subseteq I.$ Conversely, let $a$ be a monomial from
$I$. If $a \not\in I_1$, then we have the next {\em three disjoint
cases}:

{\bf Case 1.} $a \not\in (P_2 \cap K[x_1,\ldots,x_r])S$ but $a \in
(P_3 \cap K[x_1,\ldots,x_r])S.$

Let $a=uv$, where $u \in K[x_1,\ldots,x_r]$ and $v \in
K[x_{r+1},\ldots,x_n]$ monomials. From the hypothesis of this case
we get that $u \not\in (P_2 \cap K[x_1,\ldots,x_r])$. But $P_2$ is
a prime ideal, so it follows that $v \in P_2,$ which leads us to
$a \in I_3.$

{\bf Case 2.} $a \in (P_2 \cap K[x_1,\ldots,x_r])S$ but $a \not\in
(P_3 \cap K[x_1,\ldots,x_r])S.$

This case is similar with {\em Case 1.}

{\bf Case 3.} $a \not\in (P_2 \cap K[x_1,\ldots,x_r])S$ and $a
\not\in (P_3 \cap K[x_1,\ldots,x_r])S.$

Let $a=uv$, where $u \in K[x_1,\ldots,x_r]$ and $v \in
K[x_{r+1},\ldots,x_n]$ monomials. From the hypothesis of this case
we get that $u \not\in P_2 \cap K[x_1,\ldots,x_r]$ and $u \not\in
P_3 \cap K[x_1,\ldots,x_r]$. Thus $v \in P_2 \cap P_3 \cap
K[x_{r+1},\ldots,x_n]$ because $P_2$ and $ P_3$ are prime ideals.
Hence $a \in I_4$ since $u \in P_1.$

Because the cases are disjoint  we get that
the sum $I=I_1+I_2+I_3+I_4$ is direct.
\end{proof}

\begin{Proposition}\label{sdepth}
Let $P_1,P_2,P_3$ be three non-zero prime monomial ideals of S
such that there exists no inclusion between any two of them,
$\displaystyle \sum_{i=1}^{3} P_i=(x_1,\ldots,x_n)$ and set $I =
P_1 \cap P_2 \cap P_3$. With the above notations set $D=\sdepth ((I\cap K[x_1,\ldots,x_r])S)$ if $I\cap K[x_1,\ldots,x_r]\not=0$. Then
$$\sdepth\ (I) \ge  \left\{
\begin{array}{l l}
\min\{A,B,C,D\} &, \textnormal{ if } P_{i} \not\subset P_{j}+P_{k} \textnormal{ for any different }  i,j,k \in \{1,2,3\}\\
\min\{A,B,D\} &, \textnormal{ if } P_{1} \subset P_{2}+P_{3}.\\
\end{array}\right.$$
\end{Proposition}

The proof follows from Lemma \ref{sd}, but first we see the idea
in the following example:

\begin{Example}{\em
Let $S=K[x_1,x_2,x_3,x_4]$ , $P_1=(x_1,x_2)$, $P_2=(x_2,x_3,x_4)$
and $P_3=(x_1,x_3)$. Then $I=P_1 \cap P_2 \cap P_3 =
(x_1x_2,x_1x_3,x_1x_4,x_2x_3)$ and the following Stanley
decomposition of $I$ is given by Lemma \ref{sd}:
$$I=(x_1x_2)\ K[x_1,x_2,x_3,x_4]\oplus (x_2x_3)\ K[x_2,x_3,x_4] \oplus (x_1x_3,x_1x_4)\
K[x_1,x_3,x_4] .$$ Note that the
last term of Lemma \ref{sd} does not appear in this example since
$P_1 \subset P_2+P_3$. We see that the first and the third term in
the sum are principal ideals. Therefore
$\sdepth((x_1x_2)K[x_1,x_2,x_3,x_4])=4$ and
$\sdepth((x_2x_3)K[x_2,x_3,x_4])=3.$ As for the second term we use
Lemma \ref{pq}, so
$$\sdepth((x_1)\cap(x_3,x_4)\ K[x_1,x_3,x_4])=\left\lceil
\frac{1}{2}\right\rceil + \left\lceil \frac{2}{2} \right\rceil =
2.$$ Thus $\sdepth\ (I) \ge \min\{4,2,3\}=2.$ The same thing
follows from the Proposition \ref{sdepth} because in this case
$b_2=1, b_3=1, b_1=2, a_{23}=1, a_{32}=1, c=1, A=2$ and $B=3$.
Therefore $\sdepth(I) \ge \min\{A,B\}=2$. Note that $\depth\ (I) =
2$ by Proposition \ref{depth} so Stanley's Conjecture holds for
this example. }\end{Example}

{\em Proof of Proposition \ref{sdepth}}

From Lemma \ref{sd} we have the direct sum of  spaces
$I = I_1 \oplus I_2 \oplus I_3 \oplus I_4$ where:
$$I_1=(I\cap K[x_1,\ldots ,x_r])S,\ \
I_2=(P_2 \cap S')S'[x_{r+1},\ldots,x_n] \cap (P_3 \cap
S'[x_{r+1},\ldots,x_n]),$$ $$I_3=(P_3 \cap
S'')S''[x_{r+1},\ldots,x_n] \cap (P_2 \cap
S''[x_{r+1},\ldots,x_n]),\ \ I_4 = I \cap \widetilde{S}.$$ Then
$\sdepth\ (I) \ge \min_{1\le i \le 4}\sdepth (I_i)$.
 By Lemma \ref{sdep1} we have
$$\sdepth_{S'[x_{r+1},\ldots,x_n]}\ (I_2) \ge \sdepth_{S'}(P_2 \cap
S')+\sdepth_{K[x_{r+1},\ldots,x_n]}(P_3 \cap
K[x_{r+1},\ldots,x_n])=$$
$$=\left(\left\lceil\frac{a_{23}}{2}\right\rceil+r-b_3-a_{23}\right)+\left(\left\lceil\frac{\textnormal{ht}(P_3)-b_3}{2}\right\rceil+(n-r)-
(\textnormal{ht}(P_3)-b_3)\right)=$$
$$=\left\lceil\frac{a_{23}}{2}\right\rceil+\left\lceil\frac{\textnormal{ht}(P_3)-b_3}{2}\right\rceil+n-a_{23}-\textnormal{ht}(P_3)=B,$$
applying also  \cite{Sh}(see also
\cite{C}). Similarly we get

$$\sdepth_{S''[x_{r+1},\ldots,x_n]}\ (I_3) \ge
\left\lceil\frac{a_{32}}{2}\right\rceil+\left\lceil\frac{\textnormal{ht}(P_2)-b_2}{2}\right\rceil+n-a_{32}-\textnormal{ht}(P_2)=
A.$$ Finally, $$\sdepth_{\widetilde{S}}\ (I_4) \ge
\left\lceil\frac{r-b_1}{2}\right\rceil +
\sdepth_{K[x_{r+1},\ldots,x_n]}(P_2 \cap P_3 \cap
K[x_{r+1},\ldots,x_n]).$$ There are two cases:

{\bf Case 1.} If $P_{i} \not\subset P_{j}+P_{k} \textnormal{ for
any different }  i,j,k \in \{1,2,3\}$

Note that $P_2 \cap K[x_{r+1},\ldots,x_n] \not\subset P_3 \cap
K[x_{r+1},\ldots,x_n]$ (otherwise it would result that $P_2
\subset P_1+P_3$, {\em contradicting} the hypothesis of this
case). In the same idea $P_3 \cap K[x_{r+1},\ldots,x_n]
\not\subset P_2 \cap K[x_{r+1},\ldots,x_n].$ By applying Lemma
\ref{pq} for $r=\textnormal{ht}(P_2)-b_2-c$ and
$n-t=\textnormal{ht}(P_3)-b_3-c$ we get
$$\sdepth_{K[x_{r+1},\ldots,x_n]}(P_2 \cap P_3 \cap
K[x_{r+1},\ldots,x_n]) \ge \left\lceil\frac{\textnormal{ht}(P_2)-b_2-c}{2}\right\rceil+\left\lceil\frac{\textnormal{ht}(P_3)-b_3-c}{2}\right\rceil.$$
Note that there are no free variables above. Therefore
$$\sdepth_{\widetilde{S}}\ (I_4) \ge
\left\lceil\frac{r-b_1}{2}\right\rceil+\left\lceil\frac{\textnormal{ht}(P_2)-b_2-c}{2}\right\rceil+\left\lceil\frac{\textnormal{ht}(P_3)-b_3-c}{2}\right\rceil
= C.$$ Consequently, it follows
$$\sdepth \ (I) \ge \min\{A,B,C,D\}.$$

{\bf Case 2.} If $P_1 \subset P_2+P_3$

Note that in this case $\widetilde{S}=K[x_{r+1},\ldots,x_n]$ and
$P_1 \cap \widetilde{S} =0.$ Thus $I_4$ does not appear in the
Stanley decomposition of $I$ given by Lemma \ref{sd}. Hence
$\sdepth \ (I) \ge \min\{A,B,D\}.$

If one $I_i=0$, we consider in both cases that its corresponding
integer from $\{A,B,C,D\}$ will not appear in the $\sdepth$ formula.

$\Box$

\begin{Remark}{\em In the notations and hypotheses of Proposition
\ref{sdepth}, let $\widehat{S}=S[x_{n+1},\ldots,x_t]$ for some
$t>n$. Then
$$\sdepth_{\widehat{S}}\ (I\widehat{S}) \ge \left\{
\begin{array}{l l}
\min\{A,B,C,D\}+(t-n) &, \textnormal{if } P_{i} \not\subset P_{j}+P_{k} \textnormal{ for any different }i,j,k  \\
\min\{A,B,D\}+(t-n) &, \textnormal{ if } P_{1} \subset P_{2}+P_{3},\\
\end{array}\right.$$
by the Lemma \ref{hvz} and the Proposition \ref{sdepth}}.
\end{Remark}

\begin{Theorem}\label{sd>d} Let $P_1,P_2$ and $P_3$ be three
non-zero prime monomial ideals of $S$ not included one in the
other and set $I=P_1 \cap P_2 \cap P_3$. Then,
$$\sdepth\ (I) \ge \depth\ (I).$$
\end{Theorem}
\begin{proof}By \cite{HVZ} it is enough to suppose the case when
$\displaystyle\sum_{i=1}^{3} P_i =(x_1, \ldots, x_n)$. Note that
$D=\sdepth_S\ (I_1)$ is strictly greater  than the number of free
variables which is $n-r = \dim(S/P_1)$. Thus $D \ge \depth\ (I)$  (so $D$ can be  {\em omitted} below). As in
Proposition \ref{depth} and Proposition \ref{sdepth} there are two
cases:

{\bf Case 1.}$\textnormal{ If } P_{i} \not\subset P_{j}+P_{k}
\textnormal{ for any different }  i,j,k \in \{1,2,3\}$.

By Propositions \ref{depth} and \ref{sdepth} we have $\depth(I) =
3$ and we must prove that $A \ge
3$, $B \ge 3$ and $C \ge 3$.

By hypothesis of {\em Case 1} we have  $n \ge 3$ and $A \ge
\lceil\frac{1}{2}\rceil + \lceil\frac{1}{2}\rceil+1=3$. Similarly
we get $B \ge 3$. Also it follows $C \ge
\lceil\frac{1}{2}\rceil+\lceil\frac{1}{2}\rceil+\lceil\frac{1}{2}\rceil=3.$
Therefore,  $\sdepth \ (I) \ge \depth(I).$

{\bf Case 2.} There exist different $1 \le i \le j \le k \le 3 $
such that $P_i \subset P_j+P_k.$

After a possible renumbering of $(P_i)_{1 \le i \le 3}$ we may
suppose that $P_1 \subset P_2 + P_3$. In this case we show that $A,B\ge \depth(I) =
n+2-\max\{\textnormal{ht}(P_{1}+P_{2}),\textnormal{ht}(P_{1}+P_{3})\}$
 by Propositions \ref{depth} and
\ref{sdepth}. Since $P_1 \subset P_2+P_3$ we have
$$a_{32}+\textnormal{ht}(P_2)=\textnormal{ht}(P_1+P_2).$$
As the $P_i$'s are not included one in the other we get
$$\left\lceil\frac{a_{32}}{2}\right\rceil+\left\lceil\frac{\textnormal{ht}(P_2)-b_2}{2}\right\rceil+n-(a_{32}+\textnormal{ht}(P_2))
\ge 1+1+n-\textnormal{ht}(P_1+P_2),
$$ thus $A \ge \depth(I)$. Similarly it will result that $B \ge
\depth(I)$.

In conclusion, $$\sdepth\ (I) \ge \depth\ (I).$$
\end{proof}
Next we express the integers $A$, $B$, $C$ only in terms of heights of $(P_i)$, thus independently of the numbering of the variables.
\begin{Proposition}\label{sdsem}With the notations above, we get:

$$A=\left\lceil\frac{3n-\textnormal{ht}(P_1+P_2)-\textnormal{ht}(P_2+P_3)-\textnormal{ht}(P_2)}{2}\right\rceil+\left\lceil\frac{\textnormal{ht}(P_1+P_2)-\textnormal{ht}(P_1)}{2}\right\rceil,$$

$$B=\left\lceil\frac{3n-\textnormal{ht}(P_1+P_3)-\textnormal{ht}(P_2+P_3)-\textnormal{ht}(P_3)}{2}\right\rceil+\left\lceil\frac{\textnormal{ht}(P_1+P_3)-\textnormal{ht}(P_1)}{2}\right\rceil,$$
$$C=\left\lceil\frac{n-\textnormal{ht}(P_2+P_3)}{2}\right\rceil+\left\lceil\frac{n-\textnormal{ht}(P_1+P_3)}{2}\right\rceil+\left\lceil\frac{n-\textnormal{ht}(P_1+P_2)}{2}\right\rceil.$$
\end{Proposition}
\begin{proof}By definition we have $r=\textnormal{ht}(P_1)$,
$$b_2=\textnormal{ht}(P_1)+\textnormal{ht}(P_2)-\textnormal{ht}(P_1+P_2),
\ b_3=\textnormal{ht}(P_1)+\textnormal{ht}(P_3)-\textnormal{ht}(P_1+P_3),$$
$$b_1 = \textnormal{ht}(P_1)+\textnormal{ht}(P_2+P_3)-n,$$
$$a_{23}=\textnormal{ht}(P_1+P_3)+\textnormal{ht}(P_2+P_3)-\textnormal{ht}(P_3)-n,
a_{32}=\textnormal{ht}(P_1+P_2)+\textnormal{ht}(P_2+P_3)-\textnormal{ht}(P_2)-n,$$
$$c=\textnormal{ht}(P_1+P_2)+\textnormal{ht}(P_1+P_3)-\textnormal{ht}(P_1)-n,$$
and it is enough to replace them into the definition of $A,B$ and
$C$.
\end{proof}


\begin{thebibliography}{99}

\bibitem{A} J. Apel, {\em On a conjecture of R.P. Stanley; Part I – Monomial ideals,} J. Algebraic Combin. 17 (2003), 39-56.
\bibitem{BH} W. Bruns, J. Herzog, {\em Cohen Macaulay Rings, Revised edition, Cambridge,} Cambridge University Press, 1996.
\bibitem{C}  M. Cimpoeas, {\em Stanley depth of complete intersection monomial ideals,} Bull. Math. Soc. Sc. Math. Roumanie 51(99)(2008), 205-211.
\bibitem{HVZ} J. Herzog, M. Vladoiu, X. Zheng, {\em How to compute the Stanley depth of a monomial ideal,}  J.  Algebra, 322 (2009), 3151-3169.
\bibitem{Is} M. Ishaq, {\em Upper bounds for the Stanley depth}, arXiv:AC/1003.3471.
\bibitem{P} D. Popescu,  {\em An inequality between depth and Stanley depth}, Bull. Math. Soc. Sc. Math. Roumanie 52(100), (2009), 377-382.
\bibitem{PQ} D. Popescu, I. Qureshi, {\em Computing the Stanley depth}, J. Algebra, 323 (2010), 2943-2959.
\bibitem{R} A. Rauf, {\em Depth and Stanley depth of multigraded modules}, Comm.  Algebra, 38 (2010),773-784.
\bibitem{Sh} Y. Shen, {\em Stanley depth of complete intersection monomial ideals and upper-discrete partitions,} \ J.\  Algebra 321 (2009), 1285-1292.

\bibitem{S} R.P. Stanley, {\em Linear Diophantine equations and local cohomology,} Invent. Math. 68 (1982) 175-193.
\bibitem{Vi} R.\ H.\ Villarreal, {\em Monomial Algebras}, Marcel Dekker Inc.,\ New York,\ 2001.
\end{thebibliography}
\end{document}